\documentclass[12pt,draft,a4paper]{amsart}
\usepackage[english]{babel}
\usepackage{amssymb}
\usepackage{graphicx}
\usepackage{geometry}
\usepackage{hyperref}  
\hypersetup{colorlinks=true, linkcolor=blue, anchorcolor=blue, 
citecolor=red, filecolor=blue, menucolor=blue,
urlcolor=blue}

\numberwithin{equation}{section}
\newtheorem{theorem}{Theorem}[section]

\newtheorem{proposition}[theorem]{Proposition}
\newtheorem{lemma}[theorem]{Lemma}
\theoremstyle{remark}
\newtheorem{remark}{Remark}[section]

\theoremstyle{definition}
\newtheorem{definition}[theorem]{Definition}

\DeclareMathOperator{\supp}{supp}
\newcommand{\parmag}{\mathcal{D}_{A,m}}

\newcommand{\R}{{\mathbb R}}

\newcommand{\hank}{\mathcal{H}_\nu}
\newcommand{\Z}{{\mathbb Z}}
\newcommand{\C}{{\mathbb C}}

\newcommand{\solschr}{e^{it{\sqrt{H_A}}}u_0}

\def\XXint#1#2#3{{\setbox0=\hbox{$#1{#2#3}{\int}$ }
\vcenter{\hbox{$#2#3$ }}\kern-.58\wd0}}

\usepackage{epsfig}
\usepackage{subfigure}
\usepackage{tikz}
\usetikzlibrary{positioning}
\usepackage{caption}

\begin{document}
\title
[Wave and Dirac in Aharonov-Bohm magnetic fields]
{Generalized Strichartz estimates for wave and Dirac equations in Aharonov-Bohm magnetic fields}
%
%




\bibliographystyle{plain}


\begin{abstract}
We prove generalized Strichartz estimates for wave and massless Dirac equations in Aharonov-Bohm magnetic fields. Following a well established strategy to deal with scaling critical perturbations of dispersive PDEs, we make use of Hankel transform and rely on some precise estimates on Bessel functions. As a complementary result, we prove a local smoothing estimate for the Klein-Gordon equation in the same magnetic field.
\end{abstract}


\author{Federico Cacciafesta}
\address{Dipartimento di Matematica, Universit\'a degli studi di Padova, Via Trieste, 63, 35131 Padova PD, Italy}
\email{cacciafe@math.unipd.it}

\author{Zhiqing Yin}
\address{Department of Mathematics, Beijing Institute of Technology, Beijing 100081;}
\email{yinzhiqing0714@icloud.com}

\author{Junyong Zhang}
\address{Department of Mathematics, Beijing Institute of Technology, Beijing 100081; Department of Mathematics, Cardiff University, UK}
\email{zhang\_junyong@bit.edu.cn; ZhangJ107@cardiff.ac.uk}

\maketitle


\section{Introduction}

In the last years a lot of effort has been devoted to the study of dispersive properties of PDEs perturbed by {\em scaling critical} potentials. These potentials turn to be particularly hard to be dealt with, as indeed the fact that they show the same homogeneity of the differential operator prevent the use of perturbative techniques, and force to build some ``ad hoc" strategy. The most celebrated examples for what concerns Schr\"odinger and wave equations are represented by the {\em inverse square potential}, i.e. a potential of the form
$$
V(x)=\frac{a}{|x|^2}
$$
so that the Hamiltonian becomes
\begin{equation}\label{haminvsq}
H_V=-\Delta+\frac{a}{|x|^2}
\end{equation}
for some ``not too negative" constant $a$, and the {\em Aharonov-Bohm magnetic field}, that is
\begin{equation}\label{AB}
A_B:\R^2\setminus\{(0,0)\}\to\R^2,
\quad
A_B(x)=\alpha\left(-\frac{x_2}{|x|^2},\frac{x_1}{|x|^2}\right),
\quad
\alpha\in\R,
\quad
x=(x_1,x_2)
\end{equation}
so that the Hamiltonian becomes
\begin{equation}\label{eq:H}
H_A=\left(-i\nabla+\alpha\left(-\frac{x_2}{|x|^2},\frac{x_1}{|x|^2}\right)\right)^2.
\end{equation}
We refer to \cite{fanfel} and references therein for an overview of the spectral theory of this Hamiltonian in Aharonov-Bohm magnetic fields which, we point out, is a 2-dimensional model.
We do not intend to provide a detailed picture of the literature here, but we wish to mention at least \cite{burq1}-\cite{burq2}-\cite{miao}, for what concerns Strichartz estimates for both Schr\"odinger and wave equations associated to Hamiltonian \eqref{haminvsq} and \cite{fanfel} for what concerns time-decay and Strichartz estimates for the Schr\"odinger equation associated to \eqref{eq:H}. To the very best of our knowledge, no Strichartz estimates are available for the wave flow in this latter framework, except for \cite{cacfan2}, in which some local smoothing (and weighted Strichartz) are obtained.
For what concerns the Dirac equation the situation is much harder, due to the rich algebraic structure of the Dirac equation, and the only available results in this direction are provided in \cite{cacser} and \cite{cacfan}, in which local smoothing estimates are proved in the cases of, respectively, the Coulomb potential perturbation and the Aharonov-Bohm magnetic field.\vspace{0.2cm}

The purpose of this paper is to somehow combine the strategies of \cite{GZZ, miao, miao1} with the ideas of \cite{cacfan,cacfan2} to prove some generalized Strichartz estimates for the wave and massless Dirac equation in Aharonov-Bohm magnetic field. 

The wave equation we intend to deal with is thus the following
\begin{equation}\label{eq:wave}
\begin{cases}
\partial_t^2v+H_Av=0
\\
v(0,\cdot)=v_0(\cdot),
\\
\partial_tv(t,\cdot)=v_1(\cdot),
\end{cases}
\end{equation}
the solution of which is given by the formula
$$
v(t,\cdot)=\cos\left(t\sqrt{H_A}\right)v_0(\cdot)+\frac{\sin\left(t\sqrt{H_A}\right)}{\sqrt{H_A}}v_1(\cdot)
=\Re\left(e^{it\sqrt{H_A}}\right)v_0(\cdot)+\frac{\Im\left(e^{it\sqrt{H_A}}\right)}{\sqrt{H_A}}v_1(\cdot).
$$

The (massless) Dirac Hamiltonian in the Aharonov-Bohm magnetic field (in the units with $h=c=1$) is
%
\begin{equation}\label{dirham}
\mathcal{D}_A=\sigma_1(p_1+A^1)+\sigma_2(p_2+A^2)
\end{equation}
where $p_j=i\partial_j$, $\sigma_j$ are the standard Pauli matrices
\begin{equation}
\sigma_1=\left(\begin{array}{cc}0 & 1 \\1 & 0\end{array}\right),\quad
\sigma_2=\left(\begin{array}{cc}0 &-i \\i & 0\end{array}\right),\quad
\sigma_3=\left(\begin{array}{cc}1 & 0\\0 & -1\end{array}\right)
\end{equation}
and the magnetic potential $A_B(x)=(A^1(x),A^2(x))$ is given by \eqref{AB}. 
We recall that the Pauli matrices satisfy the following relations of anticommutations
\begin{equation*}
\sigma_j\sigma_k+\sigma_k\sigma_j=2\delta_{ik}\mathbb{I}_2,\quad j,k=1,2.
\end{equation*}
The Cauchy problem associated to the Hamiltonian \eqref{dirham} takes the form
\begin{equation}\label{eq:dirac}
\begin{cases}
\displaystyle
 i\partial_tu=\mathcal{D}_Au,\quad u(t,x):\mathbb{R}_t\times\mathbb{R}_x^2\rightarrow\mathbb{C}^{2}\\
u(0,x)=u_0(x).
\end{cases}
\end{equation}
We refer to \cite{cacfan} and references therein for further details on the model.\vspace{0.2cm}

Before stating our main results, let us introduce some useful notations:
\medskip

{\bf Notations.} We will denote in a standard way Lebesgue and Sobolev spaces, and with $L^p_tL^q_x=L^p(\mathbb{R}_t; L^q(\mathbb{R}^2_x))$ the mixed space-time Strichartz spaces (we will omit the dimension on the target space). With $L^p_{rdr}$ we will denote the radial part of the $L^p$ norm, that is $\|f\|_{L^p_{rdr}}^p=\int_0^\infty|f(r)|^prdr$; in the case $p=\infty$, we shall simply use the notation $L^\infty_{dr}$
. We use $[\phi_m]$, $m\in\Z$ to denote the space spanned by the vectors $\phi_m$.

Using the polar coordinates $x=r\theta$, $r\geq0$, $\theta\in\mathbb S^1$, and given a measurable function $F=F(t,x):\R\times\R^2\to\C$, we denote by
$$
\|F\|_{L^p_tL^q_{rdr}L^2_\theta}
:=
\left(\int_{-\infty}^{+\infty}\left(\int_{0}^{+\infty}\left(\int_{\mathbb S^1}|F(t,r,\theta)|^2\,d\sigma\right)^{q/2}\,rdr\right)^{p/q}\,dt\right)^{1/p},
$$
being $d\sigma$ the surface measure on the sphere. In particular, when $q=\infty$
$$
\|F\|_{L^p_tL^\infty_{dr}\,L^2_\theta}
:=
\left(\int_{-\infty}^{+\infty}\left(\sup_{r\in[0,\infty)}\left(\int_{\mathbb S^1}|F(t,r,\theta)|^2\,d\sigma\right)^{1/2}\right)^{p}\,dt\right)^{1/p}.
$$

We will denote with $\dot{H}^s_A$ the homogeneous Sobolev spaces induced by Hamiltonian \eqref{eq:H}, that is the space with the norm $\|f\|_{\dot{H}^s_A}=\| H_A^{\frac{s}2}f\|_{L^2}$. We refer again to \cite{fanfel} Section 2 for an overview of this norm.

In what follows we will systematically omit to differentiate between functions and spinors, as the meaning of each object will be clear from the contest.

\medskip 

We are now ready to state the main result of this paper.

\begin{theorem}\label{teo1}
Let $(p,q)\in [2,\infty]^2$ be such that 
\begin{equation}\label{pqrange}
\frac1p+\frac1q<\frac12, \quad \text{or}\quad (p,q)=(\infty, 2).
\end{equation}
Assume $\alpha\in\R$ such that
\begin{equation}\label{sma-eig}
\varepsilon={\rm dist}(\alpha,\Z)=\min_{m\in\Z}|m+\alpha|>0.
\end{equation}
For any $u_0, f\in  \dot H_A^s$ and $p>2$, the following Strichartz estimates hold
\begin{equation}\label{stri-w}
\|e^{it\sqrt{H_A}}u_0\|_{L^p_tL^q_{rdr}L^2_\theta}\leq
C
\|u_0\|_{\dot{H}^{s}_A}
\end{equation}
and
\begin{equation}\label{stri-D}
\|e^{it\mathcal{D}_A}f\|_{L^p_tL^q_{rdr}L^2_\theta}\leq
C
\|f\|_{\dot{H}^{s}_A}
\end{equation}
provided that $s=1-\frac1p-\frac2q$. 

\end{theorem}

\begin{remark}
The idea of using angular regularity to obtain some refined version of Strichartz estimates is not new: in particular, we stress the fact that we recover exactly the same range obtained for the free wave dynamics for radial data, see Theorem 1.3 in \cite{sterbenz}. For more results on the free equations, we refer to \cite{fangwang,jwy}.
%
\begin{center}
 \begin{tikzpicture}[scale=1]

\draw[->] (4,0) -- (8,0) node[anchor=north] {$\frac{1}{p}$};
\draw[->] (4,0) -- (4,4)  node[anchor=east] {$\frac{1}{q}$};

\draw  (4.1, -0.1) node[anchor=east] {$O$};
\draw  (7, 0) node[anchor=north] {$\frac12$};
\draw  (4, 3) node[anchor=east] {$\frac12$};
\draw  (5.5, 0) node[anchor=north] {$\frac14$};

\draw[thick] (4,3) -- (7,0);  
\draw[red,thick] (4,3) -- (5.5,0);
\draw[red, dashed,thick] (4,3) -- (7,0); 
\draw (3.9,3.15) node[anchor=west] {$A$};
\draw (6.9,0.2) node[anchor=west] {$B$};
\draw (5.5,-0.2) node[anchor=west] {C};
\draw (6,2.6) node[anchor=west] {$\frac{1}{p}+\frac{1}{q}=\frac12$};

\draw (6,1.6) node[anchor=west] {$\frac{2}{p}+\frac{1}{q}=\frac12$};

\draw (7,0) circle (0.06);

\filldraw[fill=gray!30](4,3)--(5.51,0)--(7,0); 
\filldraw[fill=gray!50](4,3)--(5.49,0)--(4,0); 


\draw[<-] (5,2.1) -- (6,2.6) node[anchor=south]{$~$};
\draw[<-] (4.86,1.3) -- (6,1.6) node[anchor=south]{$~$};

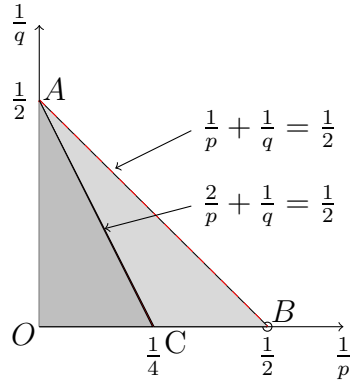
\path (6,-1.5) node(caption){Diagrammatic picture of the admissible range of $(p,q)$.};  

\end{tikzpicture}

\end{center}

\begin{remark}
For free wave and Dirac,  the classical Strichartz estimates (without loss of angular regularity) sharply hold for admissible pairs $(1/p,1/q)$ in the region AOC.
But, as far as we know,  there is no result on Strichartz estimates for wave and Dirac with Aharonov-Bohm except for \cite{fanzz}, that is a work in progress. By making use of angular regularity, we are able to cover the region AOB (except the line AB).
\end{remark}

To cover the range \eqref{pqrange}, we will only need to prove Strichartz estimates at the bottom line $OB$ (except point $B$), that is for $p>2$ and $q=\infty$, and then interpolate with the standard $L^\infty_tL^2_x$-norm estimate. We should also mention the paper \cite{mach} in which angular regularity is exploited to get ``almost" the 3D endpoint estimate (both for the wave and Dirac equations), and \cite{cacdan} in which a (small) potential perturbation is also included.
\end{remark}

\begin{remark}
As suggested by \cite{cacfan2}, with some additional technical care it would be possible to obtain Strichartz estimates for fractional Schr\"odinger equation associated to Hamiltonian \eqref{eq:H}, that is for the flow $e^{it{H_A}^{a/2}}$ for any $a>0$. Anyway, we prefer to limit our presentation here to the case $a=1$.
\end{remark}
%

\begin{remark}
It might be possible to generalize the result above (at least in the case of the wave equation) to deal with a more generic magnetic potential $A:\mathbb{S}^{N-1}\rightarrow \R^N$ in dimension $N\geq2$ satisfying the transversality condition $A(\theta)\cdot \theta=0$ for all $\theta\in\mathbb{S}^{N-1}$. Dispersive equations with potentials of these forms (and even more general ones, including also scaling critical zero-order terms) have been dealt with in literature (see e.g. \cite{fanfel}, \cite{fanzz}).
Nevertheless, as this would require a fair amount of additional technicalities and as, again, the most relevant choice from a physical point of view is given by \eqref{AB}, we prefer not to deal with the general setting.
\end{remark}

The staring role in our proof is played by the Hankel transform, which has proved in the last years to be a very effective and successful tool in the analysis of dispersive dynamics, in particular when critical perturbations come into play. The main advantage of it is in that, as we will see, it allows an explicit representation of the solution in terms of a series that involves the eigenfunctions of the operator. This fact, combined with the $L^2$ orthogonality of spherical harmonics, allows in a quite standard way to obtain estimates with angular regularity. In a nutshell, we can summarize the proof of Theorem \ref{teo1} in the following main steps
\begin{enumerate}
\item Exploit the ``spherical symmetry" of the equation to reduce to a sum of radial problems;
\item Use the Hankel transform to obtain an explicit representation of the solution;
\item Prove Strichartz estimates for frequency-localized initial data;
\item Use a dyadic decomposition and a scaling argument to obtain the final result.
\end{enumerate}
Step $(3)$ turns to be the most technically involved, and requires the use of some precise estimates on Bessel functions (see Proposition \ref{LRE}). This strategy has been strongly inspired by \cite{miao,miao1}.
Nevertheless, we should stress some significant difference with respect to \cite{miao}: in that paper the above strategy was used to {\em improve} the range of admissible exponents for Strichartz estimates for the wave equation with inverse square potentials obtained in \cite{burq1}, and those estimates were actually used in their proof. Here on the one hand, as  no Strichartz estimates are available at the moment for the solutions to \eqref{eq:wave}-\eqref{eq:dirac}, we have to prove \eqref{stri-w} directly instead of interpolating with the known Strichartz estimates. On the other hand, due to the failure of Littlewood-Paley square function inequality at the $L^\infty_{dr}$ level, we here use a different argument avoiding this to prove \eqref{stri-w} with $q=\infty$.\vspace{0.2cm}

Finally, we should mention that this same strategy is in development in \cite{cacserzha} to deal with the massless Dirac-Coulomb equation: this is definitely a much harder problem, mainly because of the fact that the generalized eigenfunctions of the Dirac-Coulomb operator enjoy a complicated representation involving confluent hypergeometric functions (instead of the Bessel ones that appear in the Aharonov-Bohm case). As a consequence, the estimates on the solution, that can be written after constructing a suitable ``relativistic Hankel transform", are quite delicate to be proved; nevertheless, a result similar to Theorem \ref{teo1} can be obtained.
\medskip

As a complementary result, we provide a local smoothing estimate for the dynamics of the Klein-Gordon equation with a magnetic field \eqref{AB}, that is for the solutions to system (we are taking $m=1$)
\begin{equation}\label{eq:kg}
\begin{cases}
\partial_t^2v+H_Av+v=0
\\
v(0,\cdot)=v_0(\cdot),
\\
\partial_tv(t,\cdot)=v_1(\cdot)
\end{cases}
\end{equation}
which is given by the formula

\begin{equation}\label{solkg}
v(t,\cdot)=\cos\left(t\sqrt{H_A+1}\right)v_0(\cdot)+\frac{\sin\left(t\sqrt{H_A+1}\right)}{\sqrt{H_A+1}}v_1(\cdot)
\end{equation}
$$
=\Re\left(e^{it\sqrt{H_A+1}}\right)v_0(\cdot)+\frac{\Im\left(e^{it\sqrt{H_A+1}}\right)}{\sqrt{H_A+1}}v_1(\cdot).
$$
Notice that the Klein-Gordon equation, due to the presence of the additional mass term, does not exhibit a scaling, and therefore some slight additional care is needed. Nevertheless, by exploiting a separate analysis of high and low frequencies, we are able to prove the following local smoothing estimate, which complements the ones for the fractional Schr\"odinger and Dirac equations obtained respectively in \cite{cacfan2} and \cite{cacfan}.

\begin{theorem}\label{teo3}
Let $v$ be a solution to \eqref{eq:kg} and let $\varepsilon$ be in \eqref{sma-eig}.
 There exists a constant $C$ such that for any $1\leq\beta<1+\varepsilon$
$$
\| |x|^{-\beta}v\|_{L^2_tL^2_x}\leq C \big(\| (1+H_A)^{\frac{2\beta-1}4}v_0\|_{L^2}  +\| (1+H_A)^{\frac{2\beta-3}4}v_1\|_{L^2}  \big).
$$
\end{theorem}

The plan of the paper is the following. Section \ref{secpre} is devoted to introduce the necessary preliminaries (overview of the spectral theory of the operators, spherical decomposition, Hankel transform and estimates on Bessel functions), while in Section \ref{secmainres} and \ref{teo3sec} we provide the proofs for our main results.

\medskip

\textbf{Acknowledgments.} The first author acknowledges support from the University of Padova STARS project ``Linear and Nonlinear Problems for the Dirac Equation" (LANPDE). The last two authors were supported by National Natural
Science Foundation of China (11771041, 11831004) and H2020-MSCA-IF-2017(790623).

\section{Preliminaries}\label{secpre}

In this section we present all the setup and the preliminaries needed to prove our results.

\subsection{Spherical decomposition, spectral theory and Hankel transform} A crucial aspect of the dynamics of dispersive equations in Aharonov-Bohm field is in that it is possible to decompose the dynamics into a sum of radial dynamics. We summarize this well known fact in the following

\begin{proposition}\label{propdec}
Let $A_B$ be given by \eqref{AB}, and let $\phi_m(\theta)=\frac{e^{im\theta}}{\sqrt{2\pi}}$ for $\theta\in[0,2\pi)$ and $m\in\Z$ be a complete orthonormal set on $L^2(\mathbb S^1)$. Then the following decompositions hold:
\begin{itemize}
\item {Laplacian decomposition.} There is a canonical isomorphism
\begin{equation*}
L^2(\R^2)\cong \bigoplus_{m\in\Z}L^2(\mathbb{R}_+,rdr)\otimes [\phi_m],
\end{equation*}
by means of the following decomposition:
\begin{equation*}
\Psi(x)=\sum_{m\in\mathbb{Z}}\frac{1}{\sqrt{2\pi}}\kappa_m(r)e^{im\theta},
\end{equation*}
where $\Psi\in L^2(\R^2)$ and $\kappa_m\in L^2(\mathbb{R}_+,rdr)$.
The action of the operator $H_A$ defined in \eqref{eq:H} with respect to the basis $\{\frac{e^{im\theta}}{\sqrt{2\pi}}\}$ is given by 
 \begin{equation}\label{H-alpha}
 \displaystyle H_{\alpha,m}=-\frac{d^2}{dr^2}-\frac1r\frac{d}{dr}+\frac{(m+\alpha)^2}{r^2},
  \end{equation} 
and $H_A$ on $C^\infty_0(\R^2)$ is unitary equivalent to the direct sum of $H_{\alpha,m}$ that is
 \begin{equation}\label{dec}
H_A= \bigoplus_{m\in\mathbb{Z}}H_{\alpha,m}.
 \end{equation} 
 
\item {Dirac decomposition.} There is a canonical isomporphism
\begin{equation*}
L^2(\mathbb{R}^2)^{2} \cong \bigoplus_{m\in\Z}L^2(\mathbb{R}_+,rdr)\otimes [h_m]
\end{equation*}
with $[h_m]=\{\phi_m,\phi_{m+1}\}$ by means of the following decomposition:
\begin{equation*}
\Phi(x)=\sum_{m\in\mathbb{Z}}\frac{1}{\sqrt{2\pi}}\left(\begin{array}{cc} f_m(r)\\
 g_m(r) e^{i\theta}
\end{array}\right)e^{im\theta},
\end{equation*}
where $\Phi\in L^2(\R^2)^2$ and $f_m, g_m\in L^2(\mathbb{R}_+,rdr)$. The action of the operator $\mathcal{D}_A$ defined in \eqref{dirham} with respect to the basis$\frac1{\sqrt{2\pi}}\{e^{im\theta},e^{i(m+1)\theta}\}$ is given by
\begin{equation}\label{partham}
\displaystyle
\parmag=\left(\begin{array}{cc}0 &-i\left(\partial_r+\frac{m+\alpha+1}r\right) \\-i\left(\partial_r-\frac{m+\alpha}r\right) & 0\end{array}\right).
\end{equation}
and $\mathcal{D}_A$ on $C^\infty_0(\R^2)^2$ is unitary equivalent to the direct sum of $\parmag$, that is
\begin{equation*}
\mathcal{D}_A\cong \bigoplus_{m\in\mathbb{Z}}\parmag.
\end{equation*}
\end{itemize}

\end{proposition}

\begin{remark}
For the sake of simplicity, from now on we will be systematically neglecting all the normalization terms involving $\pi$.
\end{remark}

\begin{proof}
See \cite{adamiteta} for the Laplacian case \cite{gerb} for the Dirac case.
\end{proof}

The spectra of both the operators $H_A$ and $\mathcal{D}_A$ are well known to be purely absolutely continuous, and in particular $\sigma(H_A)=[0,+\infty)$ and $\sigma(\mathcal{D}_A)=\R$ (we refer respectively to \cite{par} and \cite{gerb}). 
and in view of Proposition \ref{propdec} their generalized eigenfunctions can be written in terms of Bessel functions: for a fixed $m\in\mathbb{Z}$ and $E>$0 we have indeed that
\begin{equation}\label{eqspecschr}
H_{\alpha,m}\varphi_{m,E}(r)=E\varphi_{m,E}(r),
\end{equation} 
has the solution
\begin{equation}\label{contspecschr}
\varphi_{m,E}(r)\cong J_{|m+\alpha|}(Er)
\end{equation}
and
\begin{equation}\label{eqspecdirac}
\parmag\chi_{m,E}(r)=E\chi_{m,E}(r),
\end{equation} 
has the solution
\begin{equation}\label{contspecdirac}
\chi_{m,E}(r)=\left(\begin{array}{cc}f_{m,E}(r)\\g_{m,E}(r)\end{array}\right)\cong
\displaystyle
\left(\begin{array}{cc}(\epsilon_m)^mJ_{|m+\alpha|}(Er) \\ i(\epsilon_m)^{m+1}J_{|m+1+\alpha|}(Er)\end{array}\right)
\end{equation}
with
$$
\epsilon_m=\begin{cases}
1\qquad {\rm if}\:m+\alpha\geq0\\
-1\quad\:{\rm if}\:m+\alpha<0.
\end{cases}
$$
\begin{remark}
The generalized eigenfunctions for $\parmag$ for negative values of the energy can be written as
\begin{equation}\label{negen}
\chi_{m,-E}(r)=\overline{\chi_{m,E}}(r)=\left(\begin{array}{cc}(\epsilon_m)^mJ_{|m+\alpha|}(|E|r) \\ -i(\epsilon_m)^{m+1}J_{|m+1+\alpha|}(|E|r)\end{array}\right),
\end{equation}
so that in particular one has
$$
f_{m,-E}(r)=f_{m,E}(r),\qquad g_{m,-E}(r)=-g_{m,E}(r).
$$
\end{remark}

A crucial role is going to be played by the {\em Hankel transform}: we recall the definition of the standard 2-dimensional one, that  for $\nu>0$ is given by
\begin{equation}\label{hankel}
(\mathcal{H}_\nu\phi)(\xi)=\int_0^\infty J_{\nu}(r|\xi|)\phi(r\xi/|\xi|) \;rdr,
\end{equation}
that will play a leading role in the study of the wave dynamics. To deal with the Dirac equation, we need a slight algebraic manipulation of this, due to the fact that the spherical harmonics decomposition forces to work on 2-dimensional radial spaces. We therefore set the following
\begin{definition}\label{hanktras}
Let $\varphi(r)=(\varphi_1(r),\varphi_2(r))\in L^2((0,\infty),r dr)^2$.
For $m\in\mathbb{Z}$, we define the following integral transform
\begin{equation*}\label{H}
\mathcal{P}_m\varphi(E)=
\left(\begin{array}{cc}\mathcal{P}^+_m\varphi(E)\\
\mathcal{P}^-_m\varphi(E)
\end{array}\right)=
\int_0^{+\infty}H_{m}(\varepsilon r)\cdot\varphi(r)rdr
\end{equation*}
where we have introduced the matrix
\begin{equation}\label{matra}
H_{m}=\left(\begin{array}{cc}f_{m,E}(r) &g_{m,E}(r)\\
-f_{m,-E}(r)&-g_{m,-E}(r)
\end{array}\right)
\end{equation}
with $f$ and $g$ given by \eqref{contspecdirac}, so that
\begin{equation}\label{rhankel}
\mathcal{P}_m^{+}\varphi(E)\cong\int_0^\infty \big(J_{|m+\alpha|}(Er)\varphi_1(r)+J_{|m+1+\alpha|}(Er)\varphi_2(r)\big)r dr
\end{equation}
and a similar one for $\mathcal{P}_m^{-}$.
\end{definition}

The Hankel transform (both the standard one \eqref{hankel} and the ``relativistic" one \eqref{rhankel} introduced in Definition \ref{hanktras}) satisfies several important properties, including the fact that it allows to define in a quite standard way the fractional powers of the operators $H_A$ and $\mathcal{D}_A$. For $\mathcal{D}_A$, we refer to Section 2 in \cite{cacfan} and Proposition 2.3 in \cite{cacfan2}; while for $H_A$, we record for convenience and
refer the readers for analogues to M.Taylor \cite[Chapter 9]{Taylor}, (see also \cite{burq1}).
\begin{lemma}\label{lem:hankel}
Let $\mathcal{H}_\nu$ be the Hankel transform in \eqref{hankel} with $\nu=|m+\alpha|$ and $H_{\alpha,m}$ in \eqref{H-alpha}.
Then
\begin{item}
\item $\mathrm{(1)}$ $\mathcal{H}_\nu=\mathcal{H}^{-1}_\nu,$

\item $\mathrm{(2)}$ $\mathcal{H}_\nu$ is self-adjoint, i.e.$\quad \mathcal{H}_\nu=\mathcal{H}^{*}_\nu$,

\item $\mathrm{(3)}$ $\mathcal{H}_\nu$ is an $L^2$ isometry, i.e. $\|\mathcal{H}_\nu f\|_{L^2}=\|f\|_{L^2},$

\item $\mathrm{(4)}$ $\mathcal{H}_\nu(H_{\alpha,m} f)(\rho,\theta)=\rho^2(\mathcal{H}_\nu f)(\rho,\theta),$ for $f\in L^2.$

\end{item}

\end{lemma}

\subsection{Estimates on Bessel functions} 

In what follows we will make use of a number of estimates on Bessel functions: we collect them in the following two results.

\begin{proposition} Let $J_\nu(r)$ be the Bessel function of order $\nu$.
The following estimates hold true with a constant $C$ independent on $\nu$:
\begin{itemize}
\item Let $\nu>-\frac12$, then
\begin{equation}\label{2.4}
|J_\nu(r)|\leq \frac{Cr^\nu}{2^\nu\Gamma(\nu+\frac12)\Gamma(1/2)}\left(1+\frac1{\nu+1/2}\right)
\end{equation}
and
\begin{equation}\label{2.5}
|J'_\nu(r)|\leq \frac{C(\nu r^{\nu-1}+r^\nu)}{2^\nu\Gamma(\nu+\frac12)\Gamma(1/2)}\left(1+\frac1{\nu+1/2}\right).
\end{equation}
\item Let $r,\nu\gg 1$. Then
\begin{equation}\label{lem2.2}
|J'_\nu(r)|\leq \frac{C}{\sqrt{r}}.
\end{equation}
\end{itemize}
\end{proposition}

\begin{proof}
These estimates are quite standard: we refer to Section 2 in \cite{miao} and references therein.
\end{proof}

\begin{proposition}\label{lem:Bessel}  Let $\chi\in\mathcal{C}_c^\infty([0,1])$ be such that $\chi(x)\in[0,1]$ and $\chi(x)=1$ for $x\in[0,1/2]$. Let $0<\delta\ll1$. Then there exists a decomposition for the Bessel function $J_\nu(r)$:
\begin{equation}\label{dec}
J_{\nu}(r)=J_{\nu,1}(r)+J_{\nu,2}(r)+E_\nu(r)
\end{equation}
where
\begin{equation*}
\begin{split}
J_{\nu,1}(r)=\frac{1}{2\pi}\int_{-\pi}^\pi e^{ir\sin\theta-i\nu\theta} \chi(\frac{\theta}{\delta})
d\theta,\quad
J_{\nu,2}(r)&=\frac{1}{2\pi}\int_{-\pi}^\pi e^{ir\sin\theta-i\nu\theta} (1-\chi)(\frac{\theta}{\delta})
d\theta.
\end{split}
\end{equation*}
and 
\begin{equation*}
\begin{split}
E_{\nu}(r)=-\frac{\sin(\nu\pi)}{\pi}\int_0^\infty e^{-(r\sinh s+\nu s)}
ds\end{split}
\end{equation*}
Furthermore, for $r\gg1$, there exists a constant $C$ independent of $r,\nu$ such that
\begin{equation}\label{est:E'}
 |E_{\nu}(r)|+ |E'_{\nu}(r)| \leq Cr^{-1}
\end{equation}
and
\begin{equation}\label{est:J2}
 |J_{\nu,2}(r)|+ |J'_{\nu,2}(r)| \leq Cr^{-1/2}.
\end{equation}

\end{proposition}

\begin{proof} Most of these properties can be found in Watson \cite{Watson}; we provide a sketch of their proof for convenience. We use the Schl\"afli's integral representation (see \cite{Watson} pag. 177) to write  
\begin{align}\label{SIR'}
J_{\nu}(r)&=\frac{1}{2\pi}\int_{-\pi}^\pi e^{ir\sin\theta-i\nu\theta}
d\theta-\frac{\sin(\nu\pi)}{\pi}\int_0^\infty e^{-(r\sinh s+\nu s)}
ds,
\end{align}
then it follows \eqref{dec}. A direct computation gives
\begin{equation}
 |E'_{\nu}(r)|=\Big|\frac{\sin(\nu\pi)}{\pi}\int_0^\infty e^{-(r\sinh(s)+\nu s)}
\sinh(s)ds \Big|\leq C(r+\nu)^{-1}
\end{equation}
which implies \eqref{est:E'}. Now we consider \eqref{est:J2}.
Let
\begin{equation*}
\Phi_{r,\nu}(\theta)=\sin\theta-\frac{\nu}{r}\theta
\end{equation*}
and a simple computation shows the derivatives 
\begin{equation*}
\Phi'_{r,\nu}(\theta)=\cos\theta-\frac{\nu}{r}, \quad \Phi''_{r,\nu}(\theta)=-\sin\theta.
\end{equation*}
Thus, on intervals $[-\pi,-\frac\pi2-\delta]$ and $[\frac\pi 2+\delta,\pi]$, $\Phi$ is monotonic respectively and $$|\Phi'_{r,\nu}(\theta)|=|\cos\theta-\frac{\nu}{r}|=\frac{\nu}{r}+|\cos\theta|\geq\sin\delta.$$ Therefore Van der Corput lemma(see \cite[Proposition 2, Page 332]{Stein}) implies
\begin{equation}\label{est:J2-1}
\begin{split}
\frac{1}{2\pi}\Big(\int_{-\pi}^{-\frac\pi2-\delta}+\int_{\frac{\pi}2+\delta}^\pi\Big) \Big(e^{ir\sin\theta-i\nu\theta} (1-\chi)(\frac{\theta}{\delta})\Big)
d\theta \leq C_\delta r^{-1}\\
\frac{d}{dr}\left(\frac{1}{2\pi}\Big(\int_{-\pi}^{-\frac\pi2-\delta}+\int_{\frac{\pi}2+\delta}^\pi\Big) \Big(e^{ir\sin\theta-i\nu\theta} (1-\chi)(\frac{\theta}{\delta})\Big)
d\theta \right) \leq C_\delta r^{-1}.
\end{split}
\end{equation}
On the other hand, on the interval $[-\frac\pi2-\delta,-\delta]\cup[\delta, \frac\pi 2+\delta]$, we have
 $$|\Phi''_{r,\nu}(\theta)|=|\sin\theta|\geq\sin\delta.$$ We use the Van der Corput lemma again to obtain
\begin{equation}\label{est:J2-2}
\begin{split}
\frac{1}{2\pi}\Big(\int_{-\frac\pi2-\delta}^{-\delta}+\int_{\delta}^{\frac{\pi}2+\delta}\Big) \Big(e^{ir\sin\theta-i\nu\theta} (1-\chi)(\frac{\theta}{\delta})\Big)
d\theta \leq C_\delta r^{-1/2}\\
\frac{d}{dr}\left(\frac{1}{2\pi}\Big(\int_{-\frac\pi2-\delta}^{-\delta}+\int_{\delta}^{\frac{\pi}2+\delta}\Big) \Big(e^{ir\sin\theta-i\nu\theta} (1-\chi)(\frac{\theta}{\delta})\Big)
d\theta\right) \leq C_\delta r^{-1/2}.
\end{split}
\end{equation}
Noting $r\gg1$ and collecting \eqref{est:J2-1} and \eqref{est:J2-2}, it follows \eqref{est:J2}.

\end{proof}

%
%
%
%

\section{The proof Theorem \ref{teo1}}\label{secmainres}

The proofs of \eqref{stri-w} and \eqref{stri-D} are of course very similar, therefore we provide the details for the one of \eqref{stri-w} and only comment on the necessary modifications needed in order to obtain \ref{stri-D}. Also, we shall focus on the proof of the endpoint case $q=\infty$ as it is the hardest one \eqref{pqrange}, and only comment on the full range \eqref{pqrange} (see Remark \ref{rkrange}). We stress the fact that taking $q=\infty$ prevents the use of Littlewood-Paley theory: we will thus need to use a slightly different argument (as done in \cite{cacserzha}).

\medskip

Relying on proposition \ref{propdec} we start by writing, , for any $u_0\in \dot{H}_A^s(\R^2)$, 
\begin{equation}\label{indata}
u_0(x)=\sum_{m\in\mathbb Z}\kappa_m(r)\phi_m(\theta)
\end{equation}
with $\phi_m(\theta)\cong{e^{im\theta}}$. From Lemma \ref{lem:hankel}, we thus have, for $p>2$ and $\nu=|m+\alpha|>0$,
\begin{eqnarray}\label{stim1}
& &\|\solschr\|_{L^p_tL^\infty_{dr} L^2_\theta}=\|\sum_{m\in\Z}e^{it\sqrt{H_A}}\kappa_m(r)\phi_m(\theta)\|_{L^p_tL^\infty_{dr} L^2_\theta}\\
\nonumber
&=&\| \sum_{m\in\Z} \hank\left[e^{it{\rho}}\hank \kappa_m(r) \phi_m(\theta)\right]\|_{L^p_tL^\infty_{dr} L^2_\theta}
\\
\nonumber
&\leq&
\left(\sum_{m\in\Z}\|  \hank\left[e^{it\rho}g_m(\rho)\right]\|^2_{L^p_tL^\infty_{dr}\,}\right)^{1/2}
\end{eqnarray}
where we are denoting with $g_m(\rho)=\hank \kappa_m(\rho)$.
As a first step, we need a Strichartz estimate for data with localized frequencies

\begin{proposition}\label{proploc}
 Let $u_0$ as in \eqref{indata} be such that 
$\text{supp}\big(\hank \kappa_{m}\big)\subset [1,2]$
for all $m\in\Z$, and let $p>2$. Then
\begin{equation}\label{4.3}
\begin{split}
\|\solschr\|_{L^p_tL^\infty_{dr}\,L^2_\theta}\leq
C\|u_0\|_{L^2_x}.
\end{split}
\end{equation}
\end{proposition}

\begin{proof}
The proof of this result heavily relies on the following technical result:

\begin{proposition}\label{LRE}

Let $\varphi\in\mathcal{C}_c^\infty(\R)$ be supported in $I:=[1,2]$, $R>0$ be a dyadic number, $\varepsilon$ be in \eqref{sma-eig} and $\nu=\nu(m)=|m+\alpha|$ with $m\in\Z$. Then
\begin{equation}\label{stri-L}
\Big\|\Big(\sum_{m\in\Z}\Big|\mathcal{H}_{\nu}\big[e^{ it\rho}\varphi(\rho) g_m(\rho)\big](r)\Big|^2\Big)^{1/2}\Big\|_{L^p_tL^\infty_{dr}([R/2,R])}  
\end{equation}
\begin{equation}
\lesssim \Big\|\Big(\displaystyle\sum_{m\in\Z}|g_m(\rho)|^2\Big)^{1/2}\varphi(\rho)\Big\|_{L^2_{\rho d\rho}(I)}\times
\begin{cases}
R^{\frac{\varepsilon}2}
\qquad\quad R\lesssim1\\
R^{\frac1p-\frac12}
\qquad R\gg1.
\end{cases}
\end{equation}
\end{proposition}

Let us postpone for a moment the proof of Proposition \ref{LRE} and deduce from this the one of Proposition \ref{proploc}. We thus need to prove that 
\begin{equation}\label{4.10}
\begin{split}
\Big\|\Big(\sum_{m\in\Z}\big|\hank\big[e^{it\rho}g_m(\rho)\big](r)\big|^2\Big)^{\frac12}
\Big\|_{L^p_t(\R;L^\infty_{dr}(\R^+))}\leq
C\|u_0\|_{L^2_x}.
\end{split}
\end{equation}
Using the dyadic decomposition and the fact that $\ell^{2}\hookrightarrow \ell^{\infty}$,
 we can write
\begin{equation}\label{4.11}
\begin{split}
&\Big\|\Big(\sum_{m\in\Z}\big|\hank\big[e^{it\rho}g_m(\rho)\big](r)\big|^2\Big)^{\frac12}
\Big\|^2_{L^p_t(\R;L^\infty_{dr}(\R^+))}\\&\lesssim 
\Big\|\sup_{R\in2^{\Z}}\Big\|\Big(\sum_{m\in\Z}\big|\hank\big[e^{it\rho}g_m(\rho)\big](r)\big|^2\Big)^{\frac12}\Big\|_{L^\infty_{dr}([R/2,R])}\Big\|^2_{L^p_t(\R)}
\\&\lesssim
\Big\|\Big(\sum_{R\in2^{\Z}}\Big\|\Big(\sum_{m\in\Z}\big|\hank\big[e^{it\rho}g_m(\rho)\big](r)\big|^2\Big)^{\frac12}\Big\|^2_{L^\infty_{dr}([R/2,R])}\Big)^{1/2}\Big\|^2_{L^p_t(\R)}
\\&\lesssim
\sum_{R\in2^{\Z}}\sum_{m\in\Z}\Big\|\hank\big[e^{it\rho}g_m(\rho)\big](r)\Big\|^2_{L^p_t(\R;L^\infty_{dr}([R,2R]))}.
\end{split}
\end{equation}
 Thanks to Proposition \ref{LRE} we can estimate further with (notice that as $g_m$ is localized in $[1,2]$ the weight $\rho$ in the measure plays no role)
\begin{eqnarray}\label{cite}
\eqref{4.11}&\lesssim& \sum_{R\in2^{\Z}}\sum_{m\in\Z} \min\{R^{\frac1p-\frac12},R^{\frac{\varepsilon}2}\}^2
\|g_m(\rho)\|^2_{L^2_{\rho d\rho}}.
 \end{eqnarray}
As we are taking $p>2$, the summation in $R$ above turns to be convergent.
Therefore we have obtained
\begin{eqnarray*}
\eqref{4.11}&\lesssim&
\sum_{m\in\Z}
\|g_m(\rho)\|^2_{L^2_{\rho d\rho}}.
\end{eqnarray*}
Recalling the standard properties of the Hankel transform collected in Lemma \ref{lem:hankel}, we can eventually write
\begin{equation*}
\sum_{m\in\Z}
\|g_m(\rho)\|^2_{L^2_{\rho d\rho}}=
\sum_{m\in\Z}
\|\hank \kappa_m(\rho)\|^2_{L^2_{\rho d\rho}}=
\sum_{m\in\Z}
\|\kappa_m(r)\|^2_{L^2_{r d r}}=\|u_0\|^2_{L^2}
\end{equation*}
and this concludes the proof of Proposition \ref{proploc}.
\end{proof}

Let us now deduce the proof of \eqref{stri-w} from Proposition \ref{proploc}. Let $R$ and $N$ be dyadic numbers (i.e. let $R$ and $N$ be in $2^{\Z}$); by making a dyadic decomposition,  we can write, starting from \eqref{indata}-\eqref{stim1},
\begin{eqnarray*}
\|u(t,x)\|_{L^p_tL^\infty_{dr}\,L^2_\theta}^2&\leq&\sum_{m\in\mathbb Z}\left\|\sum_{N\in 2^\Z}\hank\left[e^{it{\rho}} \varphi(\frac{\rho}N)\hank \kappa_m(\rho) \right]\right\|_{L^p_tL^\infty_{dr}(\R^+)\,}^2
\\
&\leq&
\sum_{m\in\mathbb Z}\left\|\sup_{R\in2^{\Z}} \left\|\sum_{N\in 2^\Z} \hank\left[e^{it{\rho}} \varphi(\frac{\rho}N)\hank \kappa_m(\rho) \right]\right\|_{L^\infty_{dr}([R,2R])}\right\|_{L^p_t}^2.
\end{eqnarray*}
By using the fact that $\ell^2 \hookrightarrow\ell^\infty$ and the Minkowski  inequality, we further obtain 
\begin{eqnarray*}
\|u(t,x)\|_{L^p_tL^\infty_{dr}\,L^2_\theta}^2
&\leq&
\sum_{m\in\mathbb Z}\left\| \left(\sum_{R\in2^{\Z}}\left\|\sum_{N\in 2^\Z} \hank\left[e^{it{\rho}} \varphi(\frac{\rho}N)\hank \kappa_m(\rho) \right]\right\|^2_{L^\infty_{dr}\,([R,2R])}\right)^{1/2}\right\|_{L^p_t}^2
\\
&\leq&
 \sum_{m\in\mathbb Z} \sum_{R\in2^{\Z}}\left\|\sum_{N\in 2^\Z} \hank\left[e^{it{\rho}} \varphi(\frac{\rho}N)\hank \kappa_m(\rho) \right]\right\|^2_{L^p_tL^\infty_{dr}\,([R,2R])}
\\
&\leq&
\sum_{m\in\mathbb Z} \sum_{R\in2^{\Z}}\left(\sum_{N\in 2^\Z}\left\| \hank\left[e^{it{\rho}} \varphi(\frac{\rho}N)\hank \kappa_m(\rho) \right]\right\|_{L^p_tL^\infty_{dr}\,([R,2R])}\right)^2
\end{eqnarray*}
Notice that in the last inequality we have used the triangle inequality instead of Littlewood-Paley square function inequality, which fails at $L^\infty_{dr}$. 

By using a scaling argument, we finally get 
\begin{eqnarray*}
&=&
 \sum_{m\in\mathbb Z} \sum_{R\in2^{\Z}}\left(\sum_{N\in 2^\Z}N^{2-\frac1p}\left\| \hank\left[e^{it{\rho}} \varphi(\rho)\hank \kappa_m(N\rho) \right]\right\|_{L^p_tL^\infty_{dr}\,([NR,2NR])}\right)^2
\\&\leq&
 \sum_{m\in\mathbb Z} \sum_{R\in2^{\Z}}\left(\sum_{N\in 2^\Z}N^{1-\frac1p}Q(NR)\left\| \varphi(\frac\rho{N})\hank \kappa_m(\rho) \right\|_{L^2_{\rho d\rho}}
\right)^2,
\end{eqnarray*}
(we have also used Proposition \ref{LRE}), where
\begin{equation}
Q(NR)=
\begin{cases}(NR)^{\frac{\varepsilon}2},\qquad NR\lesssim 1\\ (NR)^{\frac1p-\frac12},\quad NR\gg1.
\end{cases}
\end{equation}
Due to the fact that we are taking $p>2$, we have
\begin{equation}\label{q-cond}
\frac1p-\frac12<0.
\end{equation}
Then, as $\varepsilon>0$, we see that
\begin{equation}\label{ST}
\sup_{R} \sum_{N\in2^\Z} Q(NR) <\infty,\quad \sup_{N} \sum_{R\in2^\Z} Q(NR) <\infty.
\end{equation}
Let 
\begin{equation}
A_{N,m}=N^{1 -\frac 1p}
\|(\hank f)(\rho) \varphi(\rho/N)\|_{L^2_{\rho d\rho}(\R^+)},
\end{equation} 
we use the Schur test lemma argument with \eqref{ST} in the following way:
\begin{equation}
\begin{split}
& \left(\sum_{R\in2^\Z}  \Big(\sum_{N\in2^\Z} Q(NR) A_{N,m} \Big)^2\right)^{1/2}\\
 &=\sup_{\|B_R\|_{\ell^2}\leq 1}\sum_{R\in2^\Z}  \sum_{N\in2^\Z} Q(NR)A_{N,m} B_R\\&\leq C  \left(\sum_{R\in2^\Z}\sum_{N\in2^\Z} Q(NR)|A_{N,m}|^2\right)^{1/2} \left(\sum_{R\in2^\Z}\sum_{N\in2^\Z} Q(NR)|B_R|^2\right)^{1/2}
 \\&\leq C\big( \sup_{R} \sum_{N\in2^\Z} Q(NR) \sup_{N} \sum_{R\in2^\Z} Q(NR)\big)^{1/2} \left(\sum_{N\in2^\Z}|A_{N,m}|^2\right)^{1/2} \left(\sum_{R\in2^\Z}|B_R|^2\right)^{1/2}
 \\&\leq C\left(\sum_{N\in2^\Z}|A_{N,m}|^2\right)^{1/2} .
\end{split}
\end{equation}
We have thus obtained
\begin{eqnarray*}
\|u(t,x)\|_{L^p_tL^\infty_{dr}\,L^2_\theta}^2 &\leq& \sum_{m\in\Z}\sum_{R\in2^\Z}  \Big(\sum_{N\in2^\Z} Q(NR) A_{N,m} \Big)^2
\\
&\leq& C\sum_{m\in\Z}\sum_{N\in2^\Z}|A_{N,m}|^2
\\
&=&\|u_0\|_{\dot H^{1-\frac1p}_A}^2 .
\end{eqnarray*}

And this concludes the proof. 
\medskip

\begin{remark}\label{rkrange}
As it is seen, condition $p>2$ is necessary in order to ensure convergence of the series on the right hand side of \eqref{cite}, as it is for \eqref{ST}. 
Then, by interpolation, one obtains the full range \eqref{pqrange}.
\end{remark}

\begin{remark}\label{rkdirac}
 The proof for \eqref{stri-D}, given Proposition \ref{LRE}, follows the same punchline, with minor necessary algebraic modifications. Indeed, the spectral projection in this case as introduced in definition \ref{hanktras} is $2$-dimensional, and involves the Hankel transform of two different orders. Nevertheless, with slight additional care due to the $2$-dimensional projection introduced in definition \ref{hanktras}, the proof works in the exact same way. We omit the details.
\end{remark}


To conclude with, we thus only need to provide a proof for Proposition \ref{LRE}.
\begin{proof}[Proof of Proposition \ref{LRE}]
 To prove this result, we divide into two cases $R\lesssim 1$ and $R\gg1$.
For $R\lesssim 1$, it suffices to prove
\begin{align*}
&\Big\|\Big(\sum_{m\in\Z}\Big|\int_0^\infty e^{\pm it\rho}
J_{\nu}(r\rho)g_m(\rho)\varphi(\rho)d\rho\Big|^2\Big)^{1/2}\Big\|
_{L^p_tL^{\infty}([R/2,R])} \nonumber \\
\lesssim& R^{\frac\varepsilon 2}
\Big\|\Big(\sum_{m\in\Z}|g_m(\rho) \varphi(\rho)|^2\Big)^{1/2}\Big\|_{L^2_{\rho d\rho}(I)}.
\end{align*}
Taking $\varepsilon$ in \eqref{sma-eig} (notice that $0<\varepsilon\leq 1/2$ by definition), by the Sobolev embedding $H^{\frac{1+\varepsilon}2}(\Omega)\hookrightarrow
L^\infty_{dr}(\Omega)$ with $\Omega=[R/2,R]$ and the interpolation, we have 
\begin{align*}
&\Big\|\Big(\sum_{m\in\Z}\Big|\int_0^\infty e^{\pm it\rho}
J_{\nu}(r\rho)g_m(\rho)\varphi(\rho)d\rho\Big|^2\Big)^{1/2}\Big\|
_{L^p_tL^{\infty}([R/2,R])} \nonumber \\
\lesssim& \Big\|\Big(\sum_{m\in\Z}\Big|\int_0^\infty e^{\pm it\rho}
J_{\nu}(r\rho)g_m(\rho)\varphi(\rho)d\rho\Big|^2\Big)^{1/2}\Big\|
_{L^p_t H^{\frac{1+\varepsilon}2}([R/2,R])} \nonumber \\
\lesssim& \Big\|\Big(\sum_{m\in\Z}\Big|\int_0^\infty e^{\pm it\rho}
J_{\nu}(r\rho)g_m(\rho)\varphi(\rho)d\rho\Big|^2\Big)^{1/2}\Big\|^{\frac{1-\varepsilon}2}
_{L^p_t L^{2}([R/2,R])} \\ & \times \Big\|\Big(\sum_{m\in\Z}\Big|\int_0^\infty e^{\pm it\rho}
J_{\nu}(r\rho)g_m(\rho)\varphi(\rho)d\rho\Big|^2\Big)^{1/2}\Big\|^{\frac{1+\varepsilon}2}
_{L^p_t H^{1}([R/2,R])} \nonumber \\
\lesssim& R^{\frac\varepsilon 2}
\Big\|\Big(\sum_{m\in\Z}|g_m(\rho) \varphi(\rho)|^2\Big)^{1/2}\Big\|_{L^2_{\rho d\rho}(I)}\nonumber 
\end{align*}
\emph{provided} we can prove the following estimates:

\begin{align}\label{Jest}
&\Big\|\Big(\sum_{m\in\Z}\Big|\int_0^\infty e^{\pm it\rho}
J_{\nu}(r\rho)g_m(\rho)\varphi(\rho)d\rho\Big|^2\Big)^{1/2}\Big\|
_{L^p_tL^2_{dr}([R/2,R])} \nonumber \\
\lesssim& R^{\frac12+\varepsilon}
\Big\|\Big(\sum_{m\in\Z}|g_m(\rho) \varphi(\rho)|^2\Big)^{1/2}\Big\|_{L^2_{\rho d\rho}(I)}.
\end{align}
and
\begin{align}\label{Jest'}
&\Big\|\Big(\sum_{m\in\Z}\Big|\int_0^\infty e^{\pm it\rho}
J'_{\nu}(r\rho)g_m(\rho)\varphi(\rho)\rho d\rho\Big|^2\Big)^{1/2}\Big\|
_{L^p_tL^2_{dr}([R/2,R])} \nonumber \\
\lesssim& R^{{-\frac12}+\varepsilon}
\Big\|\Big(\sum_{m\in\Z}|g_m(\rho) \varphi(\rho)|^2\Big)^{1/2}\Big\|_{L^2_{\rho d\rho}(I)}.
\end{align}

To prove \eqref{Jest}, since $p>2,$ we use the Minkowski inequality and the Hausdorff-Young inequality in $t$ variable to obtain
\begin{align*}
&\Big\|\Big(\sum_{m\in\Z}\Big|\int_0^\infty e^{\pm it\rho}
J_{\nu}(r\rho)g_m(\rho)\varphi(\rho)d\rho\Big|^2\Big)^{1/2}\Big\|
_{L^p_tL^2_{dr}([R/2,R])} \nonumber \\
\lesssim&
\Big\|\Big(\sum_{m\in\Z}\Big\|J_{\nu}(r\rho)
g_m(\rho)\varphi(\rho)\Big\|^2_{L^{p'}_{d\rho}(I)}\Big)^{1/2}\Big\|_{L^2_{dr}([R/2,R])}.
\end{align*}
Recalling \eqref{2.4} and using Stirling's formula $\Gamma(\nu+1)\sim\sqrt{\nu}(\nu/e)^\nu$, we obtain
\begin{align*}
&\Big\|\Big(\sum_{m\in\Z}\Big\|J_{\nu}(r\rho)
g_m(\rho)\varphi(\rho)\Big\|^2_{L^{p'}_{d\rho}(I)}\Big)^{1/2}\Big\|_{L^2_{dr}([R/2,R])} \nonumber \\
\lesssim&
R^{\frac12+\varepsilon}\Big\|\Big(\sum_{m\in\Z}|g_m(\rho)|^2\Big)^{1/2}
\varphi(\rho)\Big\|_{L^{p'}_{\rho d\rho}(I)},
\end{align*}
where we have used again Minkowski's inequality and the fact that $\rho\in I=[1,2]$. To deal with \eqref{Jest'}, we follow the same argument evoking this time \eqref{2.5}: this yields \eqref{stri-L} when $R\lesssim 1$ (we omit the details). 

\vspace{0.2cm}

Next we consider the case $R\gg1$: it is going to be enough to prove
\begin{align}\label{R>1}
&\Big\|\Big(\sum_{m\in\Z}\Big|\int_0^\infty J_{\nu}(r\rho)e^{-it\rho}g_m(\rho)\varphi(\rho)d\rho
\Big|^2\Big)^{1/2}\Big\|_{L^p_tL^\infty_{dr}([R/2,R])} \nonumber \\
\lesssim&
R^{\frac{1}p-\frac12}\Big\|\Big(\sum_{m\in\Z}|g_m(\rho)|^2\Big)^{1/2}\varphi(\rho)\Big\|_{L^2_{\rho d\rho}(I)}.
\end{align}

We need the following
\begin{lemma}
Assume 
\begin{equation}\label{H:Q}
|Q_{\nu}(r)|\lesssim C r^{-1/2},\qquad r\gg1.
\end{equation}
Then for $R\gg1$,
\begin{align}\label{est:Q}
&\Big\|\Big(\sum_{m\in\Z}\Big|\int_0^\infty Q_{\nu}(r\rho)e^{-it\rho}g_m(\rho)\varphi(\rho)d\rho
\Big|^2\Big)^{1/2}\Big\|_{L^p_tL^p_{dr}([R/2,R])} \nonumber \\
\lesssim&
R^{\frac{1}p-\frac12}\Big\|\Big(\sum_{m\in\Z}|g_m(\rho)|^2\Big)^{1/2}\varphi(\rho)\Big\|_{L^{p'}_{d\rho}(I)}.
\end{align}

\end{lemma}

\begin{proof}
Since $p>2,$ we use the Minkowski inequality and the Hausdorff-Young one in the $t$ variable to obtain
\begin{align*}
&\Big\|\Big(\sum_{m\in\Z}\Big|\int_0^\infty e^{\pm it\rho}
Q_{\nu}(r\rho)g_m(\rho)\varphi(\rho)d\rho\Big|^2\Big)^{1/2}\Big\|
_{L^p_tL^p_{dr}([R/2,R])} \nonumber \\
\lesssim&
\Big\|\Big(\sum_{m\in\Z}\Big\|Q_{\nu}(r\rho)
g_m(\rho)\varphi(\rho)\Big\|^2_{L^{p'}_\rho(I)}\Big)^{1/2}\Big\|_{L^p_{dr}([R/2,R])}\\
\lesssim&
R^{\frac{1}p-\frac12}\Big\|\Big(\sum_{m\in\Z}|g_m(\rho)|^2\Big)^{1/2}\varphi(\rho)\Big\|_{L^{p'}_{d\rho}(I)}.
\end{align*}

\end{proof}

We are now in position to prove \eqref{R>1}. To this aim, using \eqref{dec}, we need to prove
\begin{align}\label{R1>1}
&\Big\|\Big(\sum_{m\in\Z}\Big|\int_0^\infty J_{\nu,1}(r\rho)e^{-it\rho}g_m(\rho)\varphi(\rho)d\rho
\Big|^2\Big)^{1/2}\Big\|_{L^p_tL^\infty_{dr}([R/2,R])} \nonumber \\
\lesssim&
R^{\frac{1}p-\frac12}\Big\|\Big(\sum_{m\in\Z}|g_m(\rho)|^2\Big)^{1/2}\varphi(\rho)\Big\|_{L^2_{d\rho}(I)}.
\end{align}
and 
\begin{align}\label{R2>1}
&\Big\|\Big(\sum_{m\in\Z}\Big|\int_0^\infty \big(J_{\nu,2}(r\rho)+E_{\nu}(r\rho)\big)e^{-it\rho}g_m(\rho)\varphi(\rho)d\rho
\Big|^2\Big)^{1/2}\Big\|_{L^p_tL^\infty_{dr}([R/2,R])} \nonumber \\
\lesssim&
R^{\frac{1}p-\frac12}\Big\|\Big(\sum_{m\in\Z}|g_m(\rho)|^2\Big)^{1/2}\varphi(\rho)\Big\|_{L^2_{d\rho}(I)}.
\end{align}
We prove \eqref{R2>1} first. By the Sobolev embedding $W^{1,p}(\Omega)\hookrightarrow
L^\infty_{dr}(\Omega)$ with $\Omega=[R/2,R]$, it suffices to show \eqref{est:Q} with $Q_\nu(r)=J_{\nu,2}(r), J'_{\nu,2}(r), E_{\nu}(r)$ and $E'_{\nu}(r)$. By Proposition \ref{lem:Bessel}, 
we have verified \eqref{H:Q} hence we can deduce \eqref{est:Q}.
As a consequence, we obtain \eqref{R2>1}.

We next prove \eqref{R1>1}. For our purpose, we write the Fourier series of $g_m(\rho)$ as
\begin{equation}\label{period}
g_m(\rho)=\sum_jg_{m}^je^{i\frac{\pi}{2}\rho j}, \quad
g_m^j=\frac{1}{4}\int_0^4g_{m}(\rho)e^{-i\frac{\pi}{2}\rho j}  d\rho,
\end{equation}
so that
\begin{equation}\label{orth}
\|g_m(\rho)\|^2_{L^2_{d\rho}(I)}=\sum_j|g_m^j|^2.
\end{equation}
Let $\chi_\delta(\theta)=\chi(\theta/\delta)$ and recall $\nu=\nu(m)=|m+\alpha|$, we write
\begin{equation} \label{R1>1'}
\begin{split}
&\int_0^\infty J_{\nu,1}(r\rho)e^{-it\rho}g_m(\rho)\varphi(\rho)d\rho\\
=&\frac{1}{2\pi}\int_0^\infty e^{-it\rho}\int_{-\pi}^{\pi}e^{ir\rho\sin\theta-i\nu\theta}
\chi_\delta(\theta)\sum_jg_m^je^{i\frac{\pi}{2}\rho j}\varphi(\rho)d\rho d\theta\\
\lesssim&\sum_jg_m^j\int_{\R^2}e^{2\pi i\rho(r\sin\theta-(t-\frac{j}{4}))}\varphi(\rho)d\rho e^{-i\nu\theta}\chi_\delta(\theta)d\theta.
\end{split}
\end{equation}
Let $t_j=t-\frac{j}{4}$, we write 
\begin{align}\label{psi}
\psi_{t_j}^\nu(r)
=&\int_{\R^2}e^{2\pi i\rho(r\sin\theta-{t_j})}\varphi(\rho)d\rho e^{-i\nu\theta}\chi_\delta(\theta)d\theta\nonumber\\
=&\int_{\R}\check{\varphi}(r\sin\theta-{t_j}) e^{-i\nu\theta}\chi_\delta(\theta)d\theta.
\end{align}
Since $\check{\varphi}$ is a Schwartz function, then for any $N>0$, we have 
\begin{equation}\label{4.27}
|\check{\varphi}(r\sin\theta-{t_j})|\leq C_N(1+|r\sin\theta-{t_j}|)^{-N}.
\end{equation}
We consider two cases to study the properties of function $\psi_{t_j}^\nu(r)$.

{\bf Case 1: $|{t_j}|\geq4R.$} Since $r\leq2R\leq|{t_j}|$ and $|\theta|\leq\delta,$
we have
\begin{equation}\label{4.28}
|r\sin\theta-{t_j}|\geq|{t_j}|-r|\sin\theta|\geq\frac{1}{100}|{t_j}|
\end{equation}
and thus
\begin{equation}\label{4.29}
|\psi_{t_j}^\nu(r)|\leq C_{\delta,N}(1+|{t_j}|)^{-2N}.
\end{equation}
Therefore we obtain
\begin{equation*}
\eqref{R1>1'}\leq C_{\delta,N}R^{-N}\Big\|\Big(\sum_{m\in\Z}\Big|\sum_{j:4R\leq|t-\frac{j}{4}|}
g_m^j\Big(1+\Big|t-\frac{j}{4}\Big|\Big)^{-N}
\Big|^2\Big)^{1/2}\Big\|_{L^p_t(\R;L^\infty_{dr}(R/2,R))}.
\end{equation*}
Applying Cauchy-Schwartz's inequality to the above and then choosing $N$ large enough, we have
\begin{equation}
\begin{split}
\eqref{R1>1'}\leq &C_{\delta,N}R^{-N}\Big\|\Big(\sum_{m\in\Z}\sum_j\frac{|g_m^j|^2}{(1+|t-\frac j4|)^{N}}\Big)^{1/2} \Big\|_{L^p_t}\\
&\lesssim R^{-N}\Big\|\Big(\sum_{m\in\Z}\Big|g_m(\rho)\Big|^2\Big)^{1/2}\varphi(\rho)
\Big\|_{L^2_{d\rho}(I)}.
\end{split}
\end{equation}

{\bf Case 2:  $|{t_j}|<4R.$} We get based on ~\eqref{psi}~ and ~\eqref{4.27}~
\begin{align*}
|\psi_{t_j}^\nu(r)|\leq\frac{C_N}{2\pi}&\Big(\int_{\{\theta: |\theta|<2\delta,|r\sin\theta-{t_j}|\leq1\}}d\theta
\\\nonumber&\quad+\int_{\{\theta: |\theta|<2\delta,|r\sin\theta-{t_j}|\geq1\}}
(1+|r\sin\theta-{t_j}|)^{-N}d\theta\Big).
\end{align*}
Making the change of variables $y=r\sin\theta-{t_j}$, we further have
\begin{align}\label{4.33}
|\psi_{t_j}^\nu(r)|\leq\frac{C_N}{2\pi r}\Big(\int_{\{y:|y|\leq1\}}\,dy
+\int_{\{y:|y|\geq1\}}(1+|y|)^{-N}\,dy\Big)\lesssim r^{-1}.
\end{align}

We define the set $A=\{j\in\mathbb{Z}:|t-\frac{j}{4}|<4R\}$ for fixed $t$ and $R$. Obviously, the cardinality of $A$ is $O(R)$. Then,  from \eqref{4.33} and \eqref{R1>1'}, we obtain
\begin{align*}
&\Big\|\Big(\sum_{m\in\Z}\Big|\sum_{j\in A}g^j_{m}\psi^{\nu}_{t_j}(r)\Big|^2\Big)^{1/2}\Big\|_
{L^p_tL^\infty_{dr}([R/2,R])}  \\
\leq& C_{\delta,N}R^{-\frac{1}{2}}\Big(\sum_{m\in\Z}\sum_j|g^j_{m}|^2 \big(\int_{|t-\frac{j}4|<4R} dt\big)^{2/p}\Big)^{1/2}\\
=&C_{\delta,N}R^{\frac1p-\frac{1}{2}}\Big(\sum_{m\in\Z}\|g_m(\rho)\|^2_{L^2_{d\rho}}\Big)^{1/2}\\
\lesssim& R^{\frac1p-\frac{1}{2}}\Big\|\Big(\sum_{m\in\Z}|g_m(\rho)|^2\Big)^{1/2}\Big\|_{L^2_{d\rho}(I)}
\end{align*}
and thus the proof is concluded.

\end{proof}

\section{Proof of Theorem \ref{teo3}}\label{teo3sec}
 The proof follows the same lines as the one of Proposition 6.1 in \cite{kgconic}; we report here the main steps for the sake of completeness.

As a matter of fact, the result is an immediate consequence of the following

\begin{proposition}\label{propkg}
Let $\varphi\in C^\infty_c(\R\backslash\{0\})$ such that $\varphi(x)\in[0,1]$, that $\supp(\varphi)\subset [1/2,2]$ and that
$\sum_{j\in\Z}\varphi(2^{-j}\lambda)=1$ for $\lambda>0$. Set $\varphi_0(\lambda):=\sum_{j\leq0}\varphi(2^{-j}\lambda)$, let $v$ be a solution of \eqref{eq:kg} and let $\varepsilon$
be in \eqref{sma-eig}. Then there exists a constant $C$ such that 
\begin{itemize}
\item
for $1\leq \beta<1+\varepsilon$, 
\begin{equation}\label{lfest}
\| |x|^{-\beta}\varphi_0(\sqrt{H_A+1})v\|_{L^2_tL^2_x}\leq C \big(\| v_0\|_{L^2}  +\| v_1\|_{L^2}  \big)
\end{equation}
\item
for $1/2< \beta<1+\varepsilon$, 
\begin{equation}\label{hfest}
\| |x|^{-\beta}(1-\varphi_0)(\sqrt{H_A+1})v\|_{L^2_tL^2_x}\leq C \big(\| H_A^{\frac{2\beta-1}4}v_0\|_{L^2}  +\| H_A^{\frac{2\beta-3}4}v_1\|_{L^2}  \big)
\end{equation}
\end{itemize}
\end{proposition}

\begin{remark}
Notice that these two estimates are quite natural as the solutions to the Klein-Gordon equation behave like the Schr\"odinger ones for low frequencies and like the wave ones for high frequencies.
\end{remark}

\begin{proof}({\em of Proposition \ref{propkg}})
We start with \eqref{lfest}: let us denote with 
$$f^l=\varphi_0(\sqrt{H_A+1})f,\qquad f^h=(1-\varphi_0(\sqrt{H_A+1}))f
$$
 (to recall that we are dealing with the low-frequency case), so that (recalling \eqref{solkg})
$$
v^l=\frac12\left(e^{it\sqrt{H_A+1}}+e^{-it\sqrt{H_A+1}}\right)v_0^l+\frac1{2i}\frac{\left(e^{it\sqrt{H_A+1}}-e^{-it\sqrt{H_A+1}}\right)}{\sqrt{H_A+1}}v_1^l.
$$
For brevity, we limit to study the contribution from $v_0^l$ as the other ones follow the same argument. We decompose the initial datum $v^l_0$ as
$$v_0^l(x)=\sum_{m\in\mathbb Z}\kappa^l_m(r)\phi_m(\theta),$$
and we denote with $\tilde{k}_m^l=\hank \kappa^l_m$ with $\nu=|m+\alpha|$. Using standard functional calculus and recalling Proposition \ref{propdec}, we can then write the following representation
\begin{equation*}
e^{it\sqrt{H_A+1}}v_0^l(r,\theta)=\sum_{m\in\mathbb{Z}}e^{im\theta}\int_0^\infty e^{it\sqrt{\rho^2+1}}J_{|m+\alpha|}(r\rho)\varphi_0(\rho) \tilde{k}_m^l(\rho)\rho d\rho
\end{equation*}

 Thanks to the $L^2$ unitarity of the angular term and relying on Plancherel, we can write
\begin{eqnarray}\label{Gline}
\nonumber
\| |x|^{-\beta} e^{it\sqrt{H_A+1}}v_0^l\|^2_{L^2_tL^2_x}&=&
\sum_{m\in\mathbb{Z}}\int_0^\infty\left[\int_0^\infty\left|J_{|m+\alpha|}(r\rho)\varphi_0(\rho) \tilde{k}_m^l(\rho)\rho\right|^2\frac{\sqrt{\rho^2+1}}\rho d\rho\right] r^{1-2\beta}dr
\\
&\leq&
\nonumber
\sum_{m\in\mathbb{Z}}\sum_{j\leq 0}\int_0^\infty\left[\int_0^\infty\left|J_{|m+\alpha|}(r\rho) \tilde{k}_m^l(\rho)\rho\right|^2 \varphi^2(2^{-j}\rho)\frac{d\rho}\rho\right] r^{1-2\beta}dr
\\
&\leq&
\sum_{m\in\mathbb{Z}}\sum_{j\leq 0}\sum_{R\in2^\Z}2^{2j\beta}R^{1-2\beta}G_{m}(R,2^j)
\end{eqnarray}
where
$$
G_{m}(R,2^j)=\int_R^{2R}\left(\int_0^\infty|J_{|m+\alpha|}(r\rho) \tilde{k}_m^l(2^j\rho)|^2\varphi^2(\rho)d\rho\right)dr
$$
(notice that in the last inequality we have rescaled the variables $2^{-j}\rho\rightarrow \rho$ and $2^jr\rightarrow r$).
We can now rely on Proposition 4.2 in \cite{zhazhe} to estimate the term $G_{m}(R,2^j)$ as follows:
\begin{equation}\label{Gest}
G_{m}(R,2^j)\lesssim
\begin{cases}
R^{2|m+\alpha|+1}2^{-2j}\|\tilde{k}_m^l(2^j\rho)\varphi(2^{-j}\rho)\sqrt{\rho}\|^2_{L^2_{d\rho}},\qquad 	R\lesssim 1,\\
2^{-2j}\|\tilde{k}_m^l(2^j\rho)\varphi(2^{-j}\rho)\sqrt{\rho}\|^2_{L^2_{d\rho}},\qquad\qquad \qquad	R\gg 1.
\end{cases}
\end{equation}
Therefore, we can estimate
\begin{eqnarray*}
\eqref{Gline}\leq\sum_{m\in\mathbb{Z}}\sum_{j\leq 0}2^{2j(\beta-1)}\left(\sum_{R\in2^\Z,R\lesssim 1}R^{2(1+|m+\alpha|-\beta)}+\sum_{R\in2^\Z,R\gg 1}R^{1-2\beta}\right)\| \tilde{k}_m^l(\rho)\varphi(2^{-j}\rho)\sqrt{\rho}\|_{L^2_{d\rho}}^2.
\end{eqnarray*}
Notice that the two series in $R$ converge if we assume $\frac12<\beta<1+|m+\alpha|$: we thus eventually get
\begin{eqnarray*}
\| |x|^{-\beta} e^{it\sqrt{H_A+1}}v_0^l\|^2_{L^2_tL^2_x}&\lesssim& \sum_{m\in\mathbb{Z}}\sum_{j\leq 0}2^{2j(\beta-1)}\| \tilde{k}_m^l(\rho)\varphi(2^{-j}\rho)\sqrt{\rho}\|_{L^2}^2
 \\
 &\lesssim&\sum_{m\in\mathbb{Z}}\| \hank \kappa^l_m(\rho)\sqrt{\rho}\|_{L^2_{d\rho}}^2
 \\
 &\lesssim&
 \| v_0\|_{L^2}
\end{eqnarray*}
provided we further assume $\beta\geq1$ (notice that we have used the fact that the Hankel transform is an isometry on $L^2$). The term $v_1^l$ can be dealt with in the exact same way (notice that for low frequencies the factor $(1+H_A)^{-1/2}$ does not give any contribution), and this concludes the proof of \eqref{lfest}.

The proof of \eqref{hfest} follows the same line: with analogous calculations we get to the estimate
\begin{eqnarray*}
\| |x|^{-\beta} e^{it\sqrt{H_A+1}}v_0^h\|^2_{L^2_tL^2_x}&\lesssim&
 \sum_{m\in\mathbb{Z}}\sum_{j\leq 0}2^{2j(\beta-\frac12)}\| \hank \kappa^l_m(\rho)\varphi(2^{-j}\rho)\sqrt{\rho}\|_{L^2_{d\rho}}^2
\\
 &\lesssim&
 \| H_A^{\frac{2\beta-1}4}v_0\|_{L^2}.
\end{eqnarray*}
Noticing that the term $(1+H_A)^{-1/2} $ contributes with a factor $2^{-j}$ in estimate above, this concludes the proof of \eqref{hfest} and thus of Proposition \ref{propkg}.
\end{proof}

\end{document}